\documentclass[11pt,reqno]{amsart}

\newtheorem{proposition}{Proposition}[section]

\newtheorem{corollary}[proposition]{Corollary}
\newtheorem{theorem}[proposition]{Theorem}

\theoremstyle{definition}
\newtheorem{definition}[proposition]{Definition}
\newtheorem{example}[proposition]{Example}

\newcommand{\thlabel}[1]{\label{th:#1}}
\newcommand{\thref}[1]{Theorem~\ref{th:#1}}
\newcommand{\selabel}[1]{\label{se:#1}}

\newcommand{\prlabel}[1]{\label{pr:#1}}
\newcommand{\prref}[1]{Proposition~\ref{pr:#1}}
\newcommand{\colabel}[1]{\label{co:#1}}
\newcommand{\coref}[1]{Corollary~\ref{co:#1}}

\newcommand{\exlabel}[1]{\label{ex:#1}}

\newcommand{\delabel}[1]{\label{de:#1}}

\newcommand{\eqlabel}[1]{\label{eq:#1}}
\newcommand{\equref}[1]{(\ref{eq:#1})}

\def\lan{\langle}
\def\ran{\rangle}
\def\ot{\otimes}

\newcommand{\Cc}{\mathcal{C}}

\def\*C{{}^*\hspace*{-1pt}{\Cc}}
\def\text#1{{\rm {\rm #1}}}

\input xy
\xyoption {all} \CompileMatrices

\usepackage{amssymb}
\usepackage{color,amssymb,graphicx,amscd,amsmath}
\usepackage[colorlinks,urlcolor=blue,linkcolor=blue,citecolor=blue]{hyperref}

\begin{document}

\title[The automorphisms group and the classification of gradings]
{The automorphisms group and the classification of gradings of
finite dimensional associative algebras}

\author{G. Militaru}
\address{Faculty of Mathematics and Computer Science, University of Bucharest, Str.
Academiei 14, RO-010014 Bucharest 1, Romania}
\address{Simion Stoilow Institute of Mathematics of the Romanian Academy, P.O. Box 1-764, 014700 Bucharest, Romania}
\email{gigel.militaru@fmi.unibuc.ro and gigel.militaru@gmail.com}

\thanks{This work was supported by a grant of the Ministry of Research,
Innovation and Digitization, CNCS/CCCDI--UEFISCDI, project number
PN-III-P4-ID-PCE-2020-0458, within PNCDI III}


\maketitle

\begin{abstract}
Let $A$ be an $n$-dimensional algebra over a field $k$ and $a(A)$
its quantum symmetry semigroup. We prove that the automorphisms
group ${\rm Aut}_{\rm Alg} (A)$ of $A$ is isomorphic to the group
$U \bigl( G(a (A)^{\rm o} ) \bigl)$ of all invertible group-like
elements of the finite dual $a (A)^{\rm o}$. For a group $G$, all
$G$-gradings on $A$ are explicitly described and classified: the
set of isomorphisms classes of all $G$-gradings on $A$ is in
bijection with the quotient set $ {\rm Hom}_{\rm BiAlg} \, \bigl(
a (A) , \, k[G] \bigl)/\approx$ of all bialgebra maps $a (A) \,
\to  k[G]$, via the equivalence relation implemented by the
conjugation with an invertible group-like element of $a (A)^{\rm
o}$.
\end{abstract}

\section*{Introduction}

The quantum symmetry semigroup $a(A)$ of a finite dimensional
associative algebra $A$ over a field $k$, introduced by Manin
\cite{Manin} and Tambara \cite{Tambara}, is used to approach two
classical open problems: the description of the automorphisms
group ${\rm Aut}_{\rm Alg} (A)$ of $A$ and the classification of
all $G$-gradings on $A$, for an arbitrary group $G$. The
description of the automorphisms group of an algebra is an old
problem arising from Hilbert's invariant theory: the first notable
result obtained was the Skolem-Noether theorem \cite{skolem}.
Since then a lot a literature was devoted to this problem: see for
instance \cite{bavula, cek, SU} and their list of references. The
classification of all $G$-gradings on an algebra $A$ is a problem
formulated by Zelmanov \cite{Zelmanov} in the case that $A = {\rm
M}_n (k)$, the usual matrix algebra. The problem is interesting
but very diffficult and has generated a lot of interest. It is
also far from being solved: see \cite{BSZ, BS, SIN, GorS} and
their references. The similar question for the classification of
all $G$-gradings on Lie algebras was also studied a lot
\cite{eld}.

In this paper a new approach to the two above mentioned problems
is proposed using the power of a key object in quantum groups and
noncommutative geometry, namely the quantum symmetry semigroup
$a(A)$ \cite{Manin, Tambara} of an $n$-dimensional algebra $A$.
Let $\{e_1 := 1_A, \cdots, e_n\}$ be a $k$-basis of $A$ and
$\{\tau_{i, j}^s \, | \, i, j, s = 1, \cdots, n \}$ the structure
constants of $A$. Then, $a(A)$ is the free algebra generated by
$\{x_{si} \, | \, s, i = 1, \cdots, n, \}$ and the relations:
\begin{equation}\eqlabel{relatii2intro}
\sum_{u = 1}^n  \, \tau_{i, j}^u \, x_{au} = \sum_{s, t = 1}^n \,
\tau_{s, t}^a \, x_{si} x_{tj}, \qquad x_{a1} = \delta_{a, 1}
\end{equation}
for all $a$, $i$, $j = 1, \cdots, n$, where $\delta_{a, 1}$ is the
Kroneker symbol. Furthermore, $a(A)$ has a canonical bialgebra
structure such that the canonical map $ \eta_{A} : A \to A \ot a
(A)$, $\eta_{A} (e_i) := \sum_{s=1}^n \, e_s \ot x_{si}$, for all
$i = 1, \cdots, n$ is a right $a(A)$-comodule algebra structure on
$A$ and the pair $(a(A), \, \eta_{A})$ is the initial object of
the category of all bialgebras coacting on $A$
(\thref{univbialg}). The bialgebra $a(A)$ captures most of the
essential information of the algebra $A$ and this idea is
exploited in this paper. \thref{automorf} provides an explicit
description of a group isomorphism between the group of
automorphisms of $A$ and the group of all invertible group-like
elements of the finite dual $a(A)^{\rm o}$:
$$
{\rm Aut}_{\rm Alg} (A) \cong U \bigl(G\bigl( a (A)^{\rm o} \bigl)
\bigl).
$$
\prref{graduari} proves that for an arbitrary group $G$ there
exists an explicitly described bijection between the set of all
$G$-gradings on $A$ and the set of all bialgebra homomorphisms
$a(A) \to k[G]$. Furthermore, all $G$-gradings on $A$ are
explicitly classified and parameterized in \thref{nouaclas}: the
set $G-{\rm \textbf{gradings}}(A)$ of isomorphisms classes of all
$G$-gradings on $A$ is in bijection with the quotient set $ {\rm
Hom}_{\rm BiAlg} \, \bigl( a(A), \, k[G] \bigl)/\approx$ of all
bialgebra maps $a(A) \, \to  k[G]$ via the equivalence relation
implemented by the usual conjugation with an invertible group-like
element of $a (A)^{\rm o}$. Similar results where recently
obtained for Leibniz/Lie algebras \cite{AM2020}.

\section{Preliminaries}\selabel{prel}
All vector spaces, (bi)linear maps, associative algebras or
bialgebras are over an arbitrary field $k$ and $\ot = \ot_k$. We
shall denote by ${\rm Alg}_k$ the category of all unital
associative algebras and unital persevering morphisms. For $A \in
{\rm Alg}_k$ we denote by ${\rm Aut}_{\rm Alg} (A)$ the
automorphisms group of $A$ and $k[G]$ denotes the usual group
algebra of a group $G$. Let $G$ be a group and $A$ an algebra. We
recall that a \emph{$G$-grading} on $A$ is a vector space
decomposition $A = \oplus_{\sigma \in G} A_{\sigma}$ such that
$A_{\sigma}A_{\tau} \subseteq A_{\sigma \tau}$, for all $\sigma$,
$\tau \in G$. Two $G$-gradings $A = \oplus_{\sigma \in G} \,
A_{\sigma} = \oplus_{\sigma \in G} \, A_{\sigma} ^{'}$ on $A$ are
called \emph{isomorphic} if there exists $w \in {\rm Aut}_{{\rm
Alg}} (A)$ an algebra automorphism of $A$ such that $w
(A_{\sigma}) \subseteq A_{\sigma} ^{'}$, for all $\sigma \in G$.
Since $w$ is bijective and $A$ is $G$-graded we can prove easily
that the last condition is equivalent to $w (A_{\sigma}) =
A_{\sigma} ^{'}$, for all $\sigma \in G$, which is the condition
that usually appears in the literature in the classification of
$G$-gradings (\cite{eld}). For more details on the theory of
$G$-graded rings we refer to \cite{NO}.

For any bialgebra $H$ the set of group-like elements, denoted by
$G(H) := \{g\in H \, | \, \Delta (g) = g \ot g  \,\, {\rm and }
\,\, \varepsilon(g) = 1 \}$, is a monoid with respect to the
multiplication of $H$. We denote by $H^{\rm o}$, the finite dual
bialgebra of $H$, i.e.
$$
H^{\rm o} := \{ f \in H^* \,| \, f(I) = 0, \, {\rm for \, some \,
ideal} \,\, I \lhd H \,\, {\rm with} \,\, {\rm dim}_k (H/I) <
\infty \}
$$
It is well known (see for instance \cite[pag. 62]{radford}) that
$G(H^{\rm o}) = {\rm Hom}_{\rm Alg_k} (H, \,  k)$, the set of all
algebra homomorphisms $H\to k$. For a later use we also recall the
following well know result (\cite[Example 4.1.7]{Mont}): let $A$
be an algebra and $G$ a group. Then there exists a bijection
between the set of all $G$-gradings on $A$ and the set of all
right $k[G]$-comodule algebra structures on $A$ given such that
the $G$-grading on $A$ associated to a right $k[G]$-comodule
algebra structures $\varrho: A \to A \ot k[G]$ is given for any
$\sigma \in G$ by:
\begin{equation}\eqlabel{montgra}
A_{\sigma} = \{ a\in A \, | \, \varrho (a) = a \ot \sigma \}
\end{equation}
For unexplained notions pertaining Hopf algebras we refer the
reader to \cite{Mont, radford, Sw}.

\textbf{Tambara's construction revised.} We shall recall the
construction and the main properties of the quantum symmetry
semigroup of a finite dimensional algebra $A$. The results below
in this sections are well-known or have been recently proven in
much more general cases \cite{anagorj, anagorj2}. The following is
\cite[Theorem 1.1]{Tambara}: we shall present a short proof since
we prefer to give an explicit construction of the algebra $a(A)$,
using generators and relations implemented by the structure
constants of $A$ in the spirit of \cite[Theorem 3.2]{ana2019}.

\begin{theorem}\thlabel{adjunctie}
Let $A$ be a finite dimensional algebra over a field $k$. Then the
functor $A \ot - \, : {\rm Alg}_k \to {\rm Alg}_k$ has a left
adjoint $a(A, \, -)$.
\end{theorem}

\begin{proof}
Assume that ${\rm dim}_k (A) = n$ and fix $\{e_1, \cdots, e_n\}$ a
basis in $A$ over $k$ such that $e_1 = 1_A$, the unit of $A$. Let
$\{\tau_{i, j}^s \, | \, i, j, s = 1, \cdots, n \}$ be the
structure constants of $A$, i.e. for any $i$, $j = 1, \cdots, n$
we have:
\begin{equation}\eqlabel{const1}
e_i \, e_j  = \sum_{s=1}^n \, \tau_{i, j}^s \, e_s.
\end{equation}
Let $B$ be an arbitrary algebra and let $\{f_i \, | \, i \in I\}$
be a basis of $B$ containing $1_B$, i.e. $1_B = f_{i_0}$, for some
$i_0 \in I$. For any $i$, $j\in I$, let $B_{i,j} \subseteq I$ be a
finite subset of $I$ such that for any $i$, $j \in I$ we have:
\begin{equation}\eqlabel{const2}
f_i \, f_j  = \sum_{u \in B_{i, j}} \, \beta_{i, j}^u \, f_{u}.
\end{equation}
Let $a(A, \, B)$ be the free algebra generated by $\{x_{si} \, | s
= 1, \cdots, n, \, i\in I \}$ and the relations:
\begin{equation}\eqlabel{relatii}
\sum_{u \in B_{i, j}} \, \beta_{i, j}^u \, x_{au} = \sum_{s, t =
1}^n \, \tau_{s, t}^a \, x_{si} x_{tj}, \qquad x_{ai_0} =
\delta_{a, 1}
\end{equation}
for all $a = 1, \cdots, n$ and $i, \, j\in I$, where $\delta_{a,
1}$ is the Kroneker symbol. Using the relations \equref{relatii}
we can easily see that the map defined by
\begin{equation}\eqlabel{unitadj}
\eta_{B} : B \to A \ot a (A, \, B), \quad \eta_{B} (f_i) :=
\sum_{s=1}^n \, e_s \ot x_{si}, \quad {\rm for\,\, all}\,\, i\in
I.
\end{equation}
is an algebra homomorphism preserving the unit. Furthermore, for
any $k$-algebra $C$ the canonical map:
\begin{equation}\eqlabel{adjp}
\gamma_{B, \, C} \, : {\rm Hom}_{\rm Alg_k} \, ( a (A, \, B), \,
C) \to {\rm Hom}_{\rm Alg_k} \, (B, \, A \ot C), \quad \gamma_{B,
\, C} (\theta) := \bigl( {\rm Id}_{A} \ot \theta \bigl) \circ
\eta_{B}
\end{equation}
is bijective, that is for any algebra map $f : B \to A\ot C$ there
exists a unique algebra homomorphism $\theta : A (A, \, B) \to C$
such that the following diagram is commutative:
\begin{eqnarray} \eqlabel{diagrama10}
\xymatrix {& B \ar[r]^-{\eta_{B} } \ar[dr]_{f} & {A \ot a (A,
\, B)} \ar[d]^{ {\rm Id}_{A} \ot \theta }\\
& {} & {A \ot C}}
\end{eqnarray}
i.e. $f = \bigl( {\rm Id}_{A} \ot \theta \bigl) \circ \,
\eta_{B}$. The natural isomorphism \equref{adjp} proves that the
functor $a (A, \, - )$ is a left adjoint of $ A \ot -$.
\end{proof}

By taking $C := k$ in the bijection described in \equref{adjp} we
obtain:

\begin{corollary}\colabel{morlbz}
Let $A$ and $B$ be two algebras such that $A$ is finite
dimensional. Then the following map is bijective:
\begin{equation}\eqlabel{adjpk}
\gamma \, : {\rm Hom}_{\rm Alg_k} \, ( a (A, \, B), \, k) \to {\rm
Hom}_{\rm Alg_k} \, (B, \, A), \quad \gamma (\theta) := \bigl(
{\rm Id}_{A} \ot \theta \bigl) \circ \eta_{B}.
\end{equation}
\end{corollary}

\begin{definition} \delabel{alguniv}
Let $A$ be a finite dimensional algebra over a field $k$. The
algebra $a (A) := a (A, \, A)$ is called the \emph{quantum
symmetry semigroup} of $A$.
\end{definition}

The explicit description, using generators and relations, of
algebra $a (A)$ is the following: let $\{e_1, \cdots, e_n\}$ be a
$k$-basis of $A$ such that $e_1 = 1_A$ and let $\{\tau_{i, j}^s \,
| \, i, j, s = 1, \cdots, n \}$ be the structure constants of $A$.
Then, $a(A)$ is the free algebra generated by $\{x_{si} \, | s, i
= 1, \cdots, n, \}$ and the relations:
\begin{equation}\eqlabel{relatii2}
\sum_{u = 1}^n  \, \tau_{i, j}^u \, x_{au} = \sum_{s, t = 1}^n \,
\tau_{s, t}^a \, x_{si} x_{tj}, \qquad x_{a1} = \delta_{a, 1}
\end{equation}
for all $a$, $i$, $j = 1, \cdots, n$, where $\delta_{a, 1}$ is the
Kroneker symbol. Furthermore, the canonical map
\begin{equation}\eqlabel{unitadj2}
\eta_{A} : A \to A \ot a (A), \quad \eta_{A} (e_i) := \sum_{s=1}^n
\, e_s \ot x_{si}, \quad {\rm for\,\, all}\,\, i = 1, \cdots, n
\end{equation}
is an algebra homomorphism. The bijection given by \equref{adjp}
applied for $B:= A$ shows that the algebra $a(A)$ satisfies the
following first universal property:

\begin{corollary}\colabel{initialobj}
Let $A$ be a finite dimensional algebra. Then for any algebra $C$
and any algebra homomorphism $f : A \to A \ot C$, there exists a
unique algebra homomorphism $\theta: a(A) \to C$ such that $f =
({\rm Id}_{A} \ot \theta) \circ \eta_{A}$, i.e. the following
diagram is commutative:
\begin{eqnarray} \eqlabel{univerah}
\xymatrix {& A \ar[rr]^-{\eta_{A} } \ar[rrd]_{ f
} & {} & {A \ot a(A) } \ar[d]^{ {\rm Id}_{A} \ot \theta }\\
& {}  & {} & {A \ot C} }
\end{eqnarray}
\end{corollary}

The algebra $a(A)$ is also a bialgebra such that $(A, \eta_{A})$
is a right $a(A)$-comodule algebra (\cite{Tambara}). Moreover, we
can explicitly describe the coalgebra structure on $a(A)$ as
follows:

\begin{proposition} \prlabel{bialgebra}
Let $A$ be a $k$-algebra of dimension $n$. Then there exists a
unique bialgebra structure on $a(A)$ such that the algebra
homomorphism $\eta_{A} : A \to A \ot a(A)$ becomes a right
$a(A)$-comodule structure on $A$. More precisely, the
comultiplication and the counit on $a(A)$ are given for any $i$,
$j=1, \cdots, n$ by
\begin{equation} \eqlabel{deltaeps}
\Delta (x_{ij}) = \sum_{s=1}^n \, x_{is} \ot x_{sj} \quad {\rm
and} \quad  \varepsilon (x_{ij}) = \delta_{i, j}
\end{equation}
\end{proposition}

\begin{proof}
Consider the algebra homomorphism $f : A \to A \ot a(A)
 \ot a(A) $ defined by $f :=
(\eta_{A} \ot {\rm Id}_{a(A)}) \, \circ \, \eta_{A}$. It follows
from \coref{initialobj} that there exists a unique algebra
homomorphism $\Delta : a(A) \to a(A) \ot a(A)$ such that $({\rm
Id}_{A} \ot \Delta) \circ \eta_{A} = f$; that is, the following
diagram is commutative:
\begin{eqnarray} \eqlabel{delta}
\xymatrix {& A \ar[rr]^-{\eta_{A} } \ar[d]_{ \eta_{A} } & {} & {A
\ot a(A)} \ar[d]^{ {\rm Id}_{A} \ot \Delta }\\
& A \ot a(A) \ar[rr]_-{\eta_{A} \ot {\rm Id}_{a(A)}} & {} & {A \ot
a(A) \ot a(A) } }
\end{eqnarray}
Now, if we evaluate the diagram \equref{delta} at each $e_i$, for
$i = 1, \cdots, n$ we obtain, taking into account
\equref{unitadj2}, the following:
\begin{eqnarray*}
&& \sum_{t=1}^n \, e_t \ot \Delta (x_{ti}) = (\eta_{A} \ot {\rm
Id}) (\sum_{s=1}^n \, e_s \ot x_{si}) = \sum_{s=1}^n (
\sum_{t=1}^n \, e_t \ot x_{ts}) \ot x_{si}\\
&& = \sum_{t=1}^n \, e_t \ot (\sum_{s=1}^n x_{ts} \ot x_{si} )
\end{eqnarray*}
and hence $\Delta (x_{ti}) = \sum_{s=1}^n \, x_{ts} \ot x_{si}$,
for all $t$, $i=1, \cdots, n$. Obviously, $\Delta$ given by this
formula on generators is coassociative. In a similar fashion,
applying once again \coref{initialobj}, we obtain that there
exists a unique algebra homomorphism $\varepsilon: a(A) \to k$
such that the following diagram is commutative:
\begin{eqnarray} \eqlabel{epsilo}
\xymatrix {& A \ar[rr]^-{\eta_{A} } \ar[drr]_{ {\rm can} } & {} &
{A \ot a(A)} \ar[d]^{ {\rm Id}_{A} \ot \varepsilon }\\
& {} & {} & {A \ot k} }
\end{eqnarray}
where ${\rm can} : A \to A\ot k$ is the canonical isomorphism,
${\rm can} (x) = x \ot 1$, for all $x\in A$. If we evaluate this
diagram at each $e_t$, for $t = 1, \cdots, n$, we obtain
$\varepsilon (x_{ij}) = \delta_{i, j}$, for all $i$, $j=1, \cdots,
n$. It can be easily checked that $\varepsilon$ is a counit for
$\Delta$, thus $a(A)$ is a bialgebra. Furthermore, the
commutativity of the above two diagrams implies that the canonical
map $\eta_{A} : A \to A \ot a(A)$ defines a right $a(A)$-comodule
structure on $A$, i.e. $A$ is a right $a(A)$-comodule algebra.
\end{proof}

The pair $(a(A), \, \eta_{A})$, with the coalgebra structure
defined in \prref{bialgebra}, is called the {\it universal
coacting bialgebra of $A$}. The reason is that it fullfils the
following universal property which extends \coref{initialobj} and
shows that $(a(A), \, \eta_{A})$ is the initial object of the
category of all bialgebras coacting on $A$:

\begin{theorem}\thlabel{univbialg}
Let $A$ be a finite dimensional algebra. Then, for any bialgebra
$H$ and any algebra homomorphism $f \colon A \to A\otimes H$ which
makes $A$ into a right $H$-comodule there exists a unique
bialgebra homomorphism $\theta \colon a(A) \to H$ such that the
following diagram is commutative:
\begin{eqnarray} \eqlabel{univbialg}
\xymatrix {& A \ar[r]^-{\eta_{A}} \ar[dr]_-{f
} & {A \ot a(A)} \ar[d]^{ {\rm Id}_{A} \ot \theta }\\
& {}  & {A \ot H} }
\end{eqnarray}
\end{theorem}

\begin{proof} It follows from \coref{initialobj} there exists a unique algebra
map $\theta \colon a(A) \to H$ such that
diagram~\equref{univbialg} commutes. The proof will be finished
once we show that $\theta$ is a coalgebra homomorphism as well.
This follows by using again the universal property of $a(A)$.
Indeed, we obtain a unique algebra homomorphism $\psi \colon a(A)
\to H \otimes H$ such that the following diagram is commutative:
\begin{equation}\eqlabel{101}
\xymatrix{ A \ar[r]^-{\eta_{A} }\ar[rdd]_{\bigl({\rm Id}_{A}
\otimes\, \Delta_{H} \circ \theta \bigl)\circ \eta_{A} } & {A \otimes a(A) }\ar[dd]^{{\rm Id}_{A} \otimes \psi}  \\
{} & {} \\
{} & {A \otimes H \otimes H} }
\end{equation}
The proof will be finished once we show that $(\theta \otimes
\theta) \circ \Delta$ makes diagram~\equref{101} commutative.
Indeed, as  $f \colon A \to A \otimes H$ is a right $H$-comodule
structure, we have:
\begin{eqnarray*}
\bigl({\rm Id}_{A} \otimes\, (\theta \otimes \theta) \circ
\Delta\bigl)\circ \,\eta_{A}
&=& \bigl({\rm Id}_{A} \otimes \theta \otimes \theta \bigl)\circ \underline{\bigl({\rm Id}_{A} \otimes \Delta\bigl)\circ \, \eta_{A}}\\
&\stackrel{\equref{delta}} {=}& \bigl({\rm Id}_{A} \otimes \theta \otimes \theta \bigl)\circ (\eta_{A} \otimes {\rm Id}_{a(A)})\circ \eta_{A}\\
&=& \bigl(\underline{({\rm Id}_{A} \otimes \theta) \circ \eta_{A}}\ \otimes \theta \bigl)\circ \, \eta_{A}\\\
&\stackrel{\equref{univbialg}} {=}& \bigl(f \otimes \theta\bigl)\circ \, \eta_{A}\\
&=& (f \otimes {\rm Id}_{H})\circ \underline{({\rm Id}_{A} \otimes \theta) \circ \, \eta_{A}}\\
&\stackrel{\equref{univbialg}} {=}& \underline{(f \otimes  {\rm Id}_{H})\circ f}\\
&=& ({\rm Id}_{A} \otimes \Delta_{H}) \circ \underline{f}\\
&\stackrel{\equref{univbialg}} {=}& ({\rm Id}_{A} \otimes \Delta_{H}) \circ ({\rm Id}_{A} \otimes \theta) \circ \eta_{A}\\
&=& ({\rm Id}_{A} \otimes \Delta_{H} \circ \theta) \circ \eta_{A}
\end{eqnarray*}
as desired. Similarly, one can show that  $\varepsilon_H \, \circ
\, \theta = \varepsilon$ and the proof is now finished.
\end{proof}

We end this section with the following two questions:

\emph{1. Are there two finite dimensional non-isomorphic algebras
$A$ and $B$ such that $a(A) \cong a(B)$? (isomorphism of
bialgebras)?}

\emph{2. If $A$ and $B$ are two finite dimensional algebras, is it
true that $a (A \ot B) \cong a(A) \ot a(B)$ (isomorphism of
bialgebras)?}

\section{The automorphisms group and the classification of gradings on algebras}\selabel{sect3}

The power of the quantum symmetry semigroup of an algebra $A$ is
evidenced by its applications to the two open problems mentioned
in the introduction. The first application is related to the
description of the automorphisms group ${\rm Aut}_{{\rm Alg}} (A)$
of $A$.

\begin{theorem} \thlabel{automorf}
Let $A$ be a finite dimensional algebra with basis $\{e_1, \cdots,
e_n\}$ and let $U\bigl (G\bigl( a (A)^{\rm o} \bigl)\bigl)$ be the
group of all invertible group-like elements of the finite dual $a
(A)^{\rm o}$. Then the map defined for any $\theta \in
U\bigl(G\bigl( a(A)^{\rm o} \bigl)\bigl)$ and $i = 1, \cdots, n$
by:
\begin{equation} \eqlabel{izomono}
\overline{\gamma} : U \bigl(G\bigl( a(A)^{\rm o} \bigl) \bigl) \to
{\rm Aut}_{{\rm Alg}} (A), \qquad \overline{\gamma} (\theta) (e_i)
:= \sum_{s=1}^n \, \theta(x_{si}) \, e_s
\end{equation}
is an isomorphism of groups.
\end{theorem}

\begin{proof}
By applying \coref{morlbz} for $B:= A$ it follows that the map
$$
\gamma : {\rm Hom}_{\rm Alg_k} (a (A) , \, k) \to {\rm End}_{{\rm
Alg}} (A), \quad \gamma (\theta) = \bigl( {\rm Id}_{A} \ot \theta
\bigl) \circ \eta_{A}
$$
is bijective. Based on formula \equref{unitadj2}, it can be easily
seen that $\gamma$ takes the form given by \equref{izomono}. As we
mentioned in the preliminaries we have ${\rm Hom}_{\rm Alg_k} (a
(A) , k) = G\bigl( a (A)^{\rm o} \bigl)$. Therefore, since
$\overline{\gamma}$ is the restriction of $\gamma$ to the
invertible elements of the two monoids, the proof will be finished
once we show that $\gamma$ is an isomorphism of monoids. We
mention that the monoid structure on ${\rm End}_{{\rm Alg}} (A)$
is given by the usual composition of endomorphisms of the algebra
$A$, while $G\bigl( a(A)^{\rm o} \bigl)$ is a monoid with respect
to the multiplication of the bialgebra $a(A)^{\rm o}$, that is the
convolution product:
\begin{equation}\eqlabel{convolut}
(\theta_1 \star \theta_2) (x_{sj}) = \sum_{t=1}^n \,
\theta_1(x_{st}) \theta_2(x_{tj})
\end{equation}
for all $\theta_1$, $\theta_2 \in G\bigl( a(A)^{\rm o} \bigl)$ and
$j$, $s = 1, \cdots, n$. Now, for any $\theta_1$, $\theta_2 \in
G\bigl( a (A)^{\rm o} \bigl)$ and $j = 1, \cdots, n$ we have:
\begin{eqnarray*}
&& \bigl(\gamma(\theta_1) \circ \gamma(\theta_2) \bigl) (e_j) =
\gamma(\theta_1) \bigl( \sum_{t=1}^n \, \theta_2 (x_{tj}) e_t
\bigl) = \sum_{s, t = 1}^n \, \theta_1(x_{st}) \theta_2 (x_{tj})\,
e_s \\
&& = \sum_{s=1}^n \, \bigl( \sum_{t=1}^n \, \theta_1(x_{st})
\theta_2 (x_{tj}) \bigl) \, e_s = \sum_{s=1}^n \, (\theta_1 \star
\theta_2) (x_{sj}) \, e_s = \gamma (\theta_1 \star \theta_2) (e_j)
\end{eqnarray*}
thus, $\gamma (\theta_1 \star \theta_2) = \gamma(\theta_1) \circ
\gamma(\theta_2)$, and therefore $\gamma$ respects the
multiplication. We are left to show that $\gamma$ also preserves
the unit. Note that the unit $1$ of the monoid $G\bigl( a (A)^{\rm
o} \bigl)$ is the counit $\varepsilon_{a (A)}$ of the bialgebra $a
(A)$ and we obtain:
$$
\gamma(1) (e_i) = \gamma (\varepsilon_{a (A)}) (e_i) =
\sum_{s=1}^n \, \varepsilon_{a (A)} (x_{si}) \, e_s = \sum_{s=1}^n
\, \delta_{si} \, e_s = e_i = {\rm Id}_{A} (e_i)
$$
Thus we have proved that $\gamma$ is an isomorphism of monoids and
the proof is finished.
\end{proof}

The second application of the bialgebra $a(A)$ is related to the
problem of classification of all $G$-gradings on the algebra $A$.

\begin{proposition}\prlabel{graduari}
Let $G$ be a group and $A$ a finite dimensional algebra. Then
there exists a bijection between the set of all $G$-gradings on
$A$ and the set of all bialgebra homomorphisms $a(A) \to k[G]$.
The bijection is given such that the $G$-grading on $A =
\oplus_{\sigma \in G} \, A_{\sigma}^{(\theta)} $ associated to a
bialgebra map $\theta: a(A) \to k[G]$ is given by:
\begin{equation}\eqlabel{gradass}
A_{\sigma}^{(\theta)} := \{ x \in  A \, | \, \bigl({\rm Id}_{A}
\ot \theta \bigl) \, \circ \, \eta_{A} (x) = x \ot \sigma \}
\end{equation}
for all $\sigma \in G$.
\end{proposition}

\begin{proof} Applying \thref{univbialg} for the bialgebra $H := k[G]$
yields a bijection between the set of all bialgebra homomorphisms
$a(A) \to k[G]$ and the set of all algebra homomorphisms $f \colon
A \to A \otimes k[G]$ which makes $A$ into a right
$k[G]$-comodule. The proof is now finished since the latter set is
in bijective correspondence with the set of all $G$-gradings on
$A$ (see \equref{montgra} and \cite[Example 4.1.7]{Mont}).
\end{proof}

It is easy to see that, under the bijection given by
\prref{graduari}, the $G$-grading on $A$ associated to the trivial
bialgebra map $a(A) \to k[G], \, x \mapsto \varepsilon (x)$, is
just the trivial grading on $A$, that is $A_1 := A$ and
$A_{\sigma} := 0$, for any $1 \neq \sigma \in G$. In the next step
we shall classify all $G$-gradings on a given algebra $A$. We need
to recall one more elementary fact from Hopf algebras: if $H$ and
$L$ are two bialgebras over a field $k$ then the abelian group
${\rm Hom} (H, \, L)$ of all $k$-linear maps is an unital
associative algebra under the convolution product (\cite{Sw}):
$(\theta_1 \star \theta_2) (h) := \sum \, \theta_1 (h_{(1)})
\theta_2 (h_{(2)})$, for all $\theta_1$, $\theta_2 \in {\rm Hom}
(H, \, L)$ and $h\in H$.

\begin{definition}\delabel{conjug}
Let $G$ be a group and $A$ a finite dimensional algebra. Two
morphism of bialgebras $\theta_1, \theta_2: a(A) \to k[G]$ are
called \emph{conjugate} and denote $\theta_1 \approx \theta_2$, if
there exists $g \in U\bigl (G\bigl( a (A)^{\rm o} \bigl)\bigl)$ an
invertible group-like element of the finite dual $a (A)^{\rm o}$
such that $\theta_2 = g \star \theta_1 \star g^{-1}$, in the
convolution algebra ${\rm Hom} \bigl( a (A) , \, k[G] \bigl)$.
\end{definition}

We denote by ${\rm Hom}_{\rm BiAlg} \, \bigl( a (A) , \, k[G]
\bigl)/\approx $ the quotient set of the set of all bialgebra maps
$a (A) \to k[G]$ through the above equivalence relation and by
$\hat{\theta}$ the equivalence class of $\theta \in {\rm Hom}_{\rm
BiAlg} \, \bigl( a (A) , \, k[G] \bigl)$.

\begin{theorem} \thlabel{nouaclas}
Let $G$ be a group, $A$ a finite dimensional algebra and $G-{\rm
\textbf{gradings}}(A)$ the set of isomorphisms classes of all
$G$-gradings on $A$. Then the map
$$
{\rm Hom}_{\rm BiAlg} \, \bigl( a (A) , \, k[G] \bigl)/\approx
\,\,\, \mapsto \,\, G-{\rm \textbf{gradings}} (A), \qquad
\hat{\theta} \mapsto A^{(\theta)} := \oplus_{\sigma \in G} \,
A_{\sigma}^{(\theta)}
$$
where $A_{\sigma}^{(\theta)} = \{ x \in  A \, | \, \bigl({\rm
Id}_{A} \ot \theta \bigl) \, \circ \, \eta_{A} (x) = x \ot \sigma
\}$, for all $\sigma \in G$, is bijective.
\end{theorem}

\begin{proof}
Let $\{e_1, \cdots, e_n\}$ be a basis in $A$. Using
\prref{graduari} we obtain that for any $G$-grading $A =
\oplus_{\sigma \in G} \, A_{\sigma}$ on $A$ there exists a unique
bialgebra map $\theta: a (A) \to k[G]$ such that $A_{\sigma} =
A_{\sigma}^{(\theta)}$, for all $\sigma \in G$. It remains to
prove when two such $G$-gradings are isomorphic. Let $\theta_1$,
$\theta_2 : a (A) \to k[G]$ be two bialgebra maps and let $A =
A^{(\theta_1)} := \oplus_{\sigma \in G} \, A_{\sigma}^{(\theta_1)}
= \oplus_{\sigma \in G} \, A_{\sigma}^{(\theta_2)} =:
A^{(\theta_2)}$ be the associated $G$-gradings. It follows from
the proof of \prref{graduari} that given a $G$-grading on $A$ is
equivalent (and the correspondence is bijective) to a right
$k[G]$-comodule structure $\rho : A \to A \ot k[G]$ on $A$ such
that the right coaction $\rho$ is a morphism of algebras.
Moreover, the right coactions $\rho^{(\theta_1)}$ and
$\rho^{(\theta_2)} : A \to A \ot k[G]$ are implemented form
$\theta_1$ and $\theta_2$ using \thref{univbialg}, that is they
are given for any $j = 1, 2$ by
\begin{equation} \eqlabel{3000}
\rho^{(\theta_j)} : A \to A \ot k[G], \qquad \rho^{(\theta_j)}
(e_i) = \sum_{s=1}^n \, e_s \ot \theta_j (x_{si})
\end{equation}
for all $i = 1, \cdots, n$. It is well known in Hopf algebras that
the two $G$-gradings $A^{(\theta_1)}$ and $A^{(\theta_2)}$ are
isomorphic if and only if $(A, \, \rho^{(\theta_1)})$ and $(A, \,
\rho^{(\theta_2)})$ are isomorphic as algebras and right
$k[G]$-comodules, that is there exists $w : A \to A$ an algebra
automorphism of $A$ such that $\rho^{(\theta_2)} \, \circ w = (w
\ot {\rm Id}_{k[G]}) \, \circ \rho^{(\theta_1)}$. We apply now
\thref{automorf}: for any algebra automorphism $w : A \to A$ there
exists a unique $g \in U\bigl (G\bigl( a (A)^{\rm o} \bigl)\bigl)$
an invertible group-like element of the finite dual $a (A)^{\rm
o}$ such that $w = w_g$ is given for any $i = 1, \cdots, n$ by
\begin{equation} \eqlabel{3001}
w_g (e_i) = \sum_{s=1}^n \, g(x_{si}) \, e_s
\end{equation}
Using \equref{3000} and \equref{3001} we can easily compute that:
$$
\bigl (\rho^{(\theta_2)} \, \circ w_g \bigl) (e_i) = \sum_{a=1}^n
\, e_a \ot \bigl( \sum_{s=1}^n \, \theta_2 (x_{as}) g(x_{si})
\bigl)
$$
and
$$
\bigl( (w_g \ot {\rm Id}_{k[G]}) \, \circ \rho^{(\theta_1)} \bigl)
(e_i) = \sum_{a=1}^n \, e_a \ot  \bigl(\sum_{s=1}^n \, g(x_{as})
\theta_1 (x_{si}) \bigl)
$$
for all $i = 1, \cdots, n$. Thus, the algebra automorphism $w_g :
A \to A$ is also a right $k[G]$-comodule map if and only if
\begin{equation} \eqlabel{3002}
\sum_{s=1}^n \, g(x_{as}) \theta_1 (x_{si}) = \sum_{s=1}^n \,
\theta_2 (x_{as}) g(x_{si})
\end{equation}
for all $a$, $i = 1, \cdots, n$. Taking into account the formula
of the comultiplication on the bialgebra $a(A)$, the equation
\equref{3002} can be rewritten in a compact form as $(g \star
\theta_1) ( x_{ai}) = (\theta_2 \star g ) ( x_{ai})$, for all $a$,
$i = 1, \cdots, n$ in the convolution algebra ${\rm Hom} \bigl(a
(A) , \, k[G] \bigl)$, or (since $\{x_{ai}\}_{a, i = 1, \cdots,
n}$ is a system of generators of $a (A)$) just as $g \star
\theta_1 = \theta_2 \star g$. We also note that $g: a (A) \to k$
is an invertible group-like element, that is it is an invertible
element in the above convolution algebra.

To conclude, we have proved that two $G$-gradings $
A^{(\theta_1)}$ and $A^{(\theta_2)}$ on $A$ associated to two
bialgebra maps $\theta_1$, $\theta_2 : a (A) \to k[G]$ are
isomorphic if and only if there exists $g \in U\bigl (G\bigl( a
(A)^{\rm o} \bigl)\bigl)$ such that $\theta_2 = g \star \theta_1
\star g^{-1}$, in the convolution algebra ${\rm Hom} \bigl( a (A)
, \, k[G] \bigl)$, that is $\theta_1$ and $\theta_2$ are
conjugated and this finished the proof.
\end{proof}

\thref{automorf} and \thref{nouaclas} offer a theoretical answer
to the two classic problems, reducing them to typical problems for
Hopf algebras. While the explicit description of the bialgebra
$a(A)$ is relatively easy (albeit using a laborious computation,
eliminating the redundant relations among the $n^3$ relations
\equref{relatii2intro} that define it), the really difficult part
is the description of the finite dual $a (A)^{\rm o}$ using
generators and relations. Unfortunately, the explicit description
of the finite dual of a finite generated bialgebra is a problem
that has not been studied until now. It does however merit
attention in the future; the first steps in this direction have
been taken very recently in \cite{cou, Ge, LiLiu}.

\begin{example} \exlabel{a(A)1}
Let $A := k[X]/(X^2)$. Then $a \bigl(k[X]/(X^2) \bigl)$ is the
algebra $k \lan x, y \, | \, x^2 = 0, \, xy + yx = 0 \ran$ with
the coalgebra structure given by:
$$
\Delta (x) = x \ot y + 1 \ot x, \, \, \Delta(y) = y \ot y, \quad
\varepsilon (x) = 0, \,\, \varepsilon(y) = 1
$$
Indeed, let $\{e_1 := 1, \, e_2 := x\}$ be a basis of $A$, where
$x$ is the class of $X$ in $A$. Then the only non-zero structure
constants of the algebra $A$ are $\tau_{1, j}^j = \tau_{j, 1}^j =
1$, for $j = 1, 2$. Now, it is just a routine computation, to
reduce the eight relations \equref{relatii2} defining the algebra
$a \bigl(k[X]/(X^2) \bigl)$ to:
$$
x_{12}^2 = 0, \qquad x_{12} x_{22} +  x_{22}x_{12} = 0
$$
Denoting $x_{12} = x$ si $x_{22} = y$ the conclusion follows (the
coalgebra structure arises from \prref{bialgebra}). Applying
\thref{automorf} we obtain that ${\rm Aut}_{{\rm Alg}}
\bigl(k[X]/(X^2) \bigl)$ is in bijection with the set of all
morphisms of algebras $\vartheta : a \bigl(k[X]/(X^2) \bigl) \to
k$ such that the matrix $(\vartheta (x_{ij})) \in {\rm M}_2 (k)$
is invertible. Such a morphism is defined by: $\vartheta (x_{11})
= \vartheta (1): = 1$, $\vartheta (x_{21}) = \vartheta (0) := 0$,
$\vartheta (x_{12}) := \alpha$ and $\vartheta (x_{22}) := \beta$,
for some $\alpha, \beta \in k$. Since, $x_{12}^2 = 0$ we obtain
that $\alpha = 0$ and $\beta \neq 0$ since the matrix $(\vartheta
(x_{ij}))$ is invertible. Thus, ${\rm Aut}_{{\rm Alg}}
\bigl(k[X]/(X^2) \bigl) \, \cong \, k^*$.
\end{example}

\begin{example} \exlabel{a(A)2}
Let $ A: = \begin{pmatrix}
k & k \\
0 & k
\end{pmatrix} \subseteq {\rm M}_2 (k)
$ be the algebra of upper-triangular matrices with the basis
$\{e_1 := I_2, \, e_2 := e_{21}, \,  e_3: = e_{22} \}$. Then
$a(A)$ is the algebra generated by $\{x_{ij} \, | \, i, j = 1, 2,
3 \}$ and the relations:
\begin{eqnarray*}
&& x_{11} =1, \,\,  x_{21} = x_{31} = 0, \,\,  x_{12}^2 = 0, \,\,
x_{12}x_{13} = x_{12}, \,\,  x_{13}x_{12} = 0 \\
&& x_{13}^2 = x_{13}, \,\,  x_{12} x_{22} + x_{22} x_{12} = -
x_{22} x_{32}, \,\,  x_{12} x_{23} + x_{23} x_{13} + x_{22} x_{33}
= x_{22} \\
&& x_{13} x_{22} + x_{23} (x_{12} + x_{32}) = 0, \,\, x_{13}
x_{23} + x_{23} (x_{13} + x_{33}) = x_{23}\\
&& x_{32} x_{33} + x_{32}x_{13} + x_{12} x_{33} = x_{32}, \,\,
x_{32}^2 + x_{32}x_{12} + x_{12} x_{32} = 0
\end{eqnarray*}

Indeed, the only non-zero structure constants of the algebra $A$
are:
$$
\tau_{1, j}^j = \tau_{j, 1}^j = \tau_{2, 3}^2 = \tau_{3, 3}^3 := 1
$$
for all $j = 1, 2, 3$. A routine computation shows that out of the
27 relations \equref{relatii2} defining algebra $a(A)$ only the
above remain.
\end{example}

\end{document}